\documentclass[12pt,twoside]{article}
\usepackage[left=2cm,right=2cm,top=2cm,bottom=3cm]{geometry}
\usepackage[utf8]{inputenc}
\usepackage{amsmath}
\usepackage{amssymb}
\usepackage{mathrsfs}
\usepackage{enumitem}
\usepackage{amsthm}
\usepackage[framemethod=tikz,hidealllines=true,backgroundcolor=blue!20]{mdframed} 
\usepackage{float}
\usepackage{authblk} 

\usepackage[giveninits=true,maxcitenames=1]{biblatex}
\DeclareNameAlias{author}{given-family}

\title{definitions}

\newtheorem{lemma}{\textit{Lemma}}
\newtheorem{theorem}{Theorem}
\newtheorem{definition}{Definition}

\newtheorem{remark}{Remark}
\newtheorem{corollary}{Corollary}

\begin{document}
\title{\vspace{-2cm}Two extremal problems in the light of Lex graphs}
\author{Kristina Dedndreaj\thanks{This research has been funded by the European Social Fund (ESF). SAB PRANO Nr: 1562608885160.} }
\affil{University of Applied Sciences
Mittweida}

\date{\today}
\maketitle

\begin{abstract}
  Extremal problems involving independent sets are much studied. Two of the most important extremal problems in this context are concerned with the sharp upper bounds for the number of independent sets of fixed size and the independence number. In literature, these sharp upper bounds are derived in completely different contexts. In this paper, we show that both of these sharp upper bounds can be derived by considering Lex graphs. More exactly, they depend on two parameters of a sequence defined on them.
\end{abstract}

\section{Introduction}
By complementation, counting independent sets is equivalent to clique counting. The study of the latter was greatly influenced by the discovery of the  Tur\'{a}n's theorem \cite{turan1941external}. For different aspects of the study of the clique counting the reader is referred to \cite{erdos1962number}, \cite{bollobas1976complete}, \cite{moon1968independent}, \cite{lovasz1983number}, \cite{moon1965cliques}, \cite{wood2007maximum} and  \cite[ch.~VI]{bollobas2004extremal}.

We will focus on the sharp upper bounds for the number of independent sets of fixed size and the independence number of $G(n,m)$ graphs, that is, graphs of order $n$ and size $m$.
 Hanani  and Erd\"{o}s \cite{erdos1962number} proved that the maximal number of complete subgraphs of order $r$ contained in a $G(n,m)$ graph depends only on $m$.

\begin{theorem}[\cite{erdos1962number}]\label{erdos}
    Let $G(m)$ be a graph on $m$ edges and let $r\geq 3$ be a natural number. Let $s,t$ be two natural numbers such that $m=\binom{s}{2}+t$ with $0<t\leq s$. Then
    \begin{equation}
        c_r\leq \binom{s}{r} + \binom{t}{r-1}.
    \end{equation}
    Where $c_r$ denotes the number of cliques of order $r$ contained in $G(m)$.
\end{theorem}

\begin{remark}\label{svalue}
    For given $m$ the numbers $s$ and $t$ in Theorem \ref{erdos} are unique. This implies that $s$ is the smallest positive integer which satisfies
    
    \begin{equation}\label{secondin}
        m\leq \dfrac{s(s+1)}{2}.
    \end{equation}
    Therefore we have
    \begin{equation}\label{sol}
        s=\left\lceil\sqrt{\frac{1}{4}+2m}-\frac{1}{2}\right\rceil
    \end{equation}
    from where the value of $t$ is evident.
\end{remark}

\begin{lemma}\label{lemma}
Let $x$ be a real number and let $\{x\}=x-\lfloor x\rfloor$ denote its fractional part. The following holds
\begin{equation}
    \left\lceil x-\frac{1}{2}\right\rceil=
    \begin{cases}
     \left\lfloor x+\dfrac{1}{2}\right\rfloor-1 &\text{if } \{x\}=\dfrac{1}{2},\\[+4mm]
    \left\lfloor x+\dfrac{1}{2}\right\rfloor &\text{otherwise}.
    \end{cases}
\end{equation}
\end{lemma}

Note that when we want to find the maximal number of independent sets of size $r$ by using Theorem \ref{erdos} then we need to find the maximal number of cliques of size $r$ in its complement. Thus, the maximal number of independent sets of fixed size depends on the order and the size of the graph, unlike the maximal number of cliques of fixed size. We re-write Theorem \ref{erdos} in terms of independent sets.

\begin{theorem}\label{indbound}
    Let $G(n,m)$ be a graph on $n$ vertices and $m$ edges and let $r\geq 3$ be a natural number. Let $s,t$ be two natural numbers such that $\binom{n}{2}-m=\binom{s}{2}+t$ with $0<t\leq s$. Then
    \begin{equation}\label{upperbound}
        i_r\leq \binom{s}{r}+ \binom{t}{r-1}.
    \end{equation}
    Where $i_r$ denotes the number of independent sets of order $r$ contained in $G(n,m)$.
\end{theorem}

By putting $\binom{n}{2}-m$ instead of $m$ in Equation \eqref{sol} the value of $s$ in Theorem \ref{indbound} is

\begin{equation}
 \left\lceil -\frac{1}{2} +\sqrt{\frac{1}{4}+n^2-n-2m}\right\rceil.
\end{equation}

For notation convenience, in the following Theorem we will let ${\alpha}_u$ denote the sharp upper bound for the independence number of $G(n,m)$ graphs. As we will see later, the number $s$ has the following property.
\begin{theorem}\label{sproperty}
    For a given $G(n,m)$ graph let the numbers $s,t$ be as in Theorem \ref{indbound}. The followig holds
    \begin{equation}
        s=
        \begin{cases}
        {\alpha}_u-1  &\text{if } \binom{n}{2}-m=\binom{s}{2}+s,\\
        {\alpha}_u &\text{otherwise }.
        \end{cases}
    \end{equation}
\end{theorem}

It seems that this interesting fact was not proved nor noticed by Hanani  and Erd\"{o}s in \cite{erdos1962number} where Theorem \ref{erdos} makes its first appearance. However, as we will show later, it can be understood and explained in the context of Lex graphs. In the next section, we will prove Theorem \ref{indbound} and Theorem \ref{sproperty} by using Lex graphs.

\section{Lex graphs}
\begin{definition}[\cite{cutler2011extremal}]
    Given $A,B\subset \mathbf{N}$ we say $A$ precedes $B$ in lexicographic (or lex) ordering, written $A{<}_L B$, if $\min\{A\triangle B\}\in A$.
\end{definition}

The lex graph, denoted $L(n,m)$, is the graph with vertex set $[n]=\{1,\ldots,n\}$ and edge set the first $m$ elements of $\binom{[n]}{2}$ under the lex ordering. The first few elements of the lex order on $\binom{[n]}{2}$ are
\begin{equation*}
\{1, 2\} , \{1, 3\} , \{1, 4\} ,\ldots, \{1, n\} , \{2, 3\} , \{2, 4\} ,\ldots , \{2, n\} , \{3, 4\} ,\ldots
\end{equation*}

 \begin{figure}[H]
    \centering
     \begin{tikzpicture}[
            V/.style = {
                        circle,thick,fill=white},scale=0.4]
            \node[V] (1) at (6,4) [scale=0.5,draw] {1};
            \node[V] (2) at (1,4) [scale=0.5,draw] {2};
            \node[V] (3) at (3,2) [scale=0.5,draw] {3};
            \node[V] (4) at (3,6) [scale=0.5,draw] {4};
            \node[V] (5) at (9,4) [scale=0.5,draw] {5};
            \draw[black,very thick]
            (1) to (2)
            (1) to (3)
            (1) to (4)
            (1) to (5)
            (2) to (3)
            (2) to (4);
    \end{tikzpicture}
    \caption{$L(5,6)$}
    \label{L(5,6)}
    \end{figure}

The class of the lex graphs is interesting because they are extremal for the number of the independent sets. The lex graph $L(n,m)$  maximizes both, the total number independent sets and the number of independent sets of fixed size for $G(n,m)$ graphs \cite{cutler2011extremal}.

\begin{lemma}\label{exp1lemma}\cite{dedndreaj2021activities}
    Let $m,n$ be two natural numbers such that $1\leq m\leq \binom{n}{2} $. Then there exists a unique sequence ${\{p_i\}_{i=1}^k}$ such that
\begin{equation}\label{expansion1}
    \sum \limits_{i=1}^k p_i=m
\end{equation}
where:
    \begin{enumerate}
        \item $p_i>0$ and $1 \leq k \leq n-1$,
        \item $p_i=n-i$ for $i<k$,
        \item If $p_k\neq n-k$ then $p_k\in\{1,\ldots,n-k-1\}$.
    \end{enumerate}
\end{lemma}
\begin{proof}
The proof can be done by induction on $m$.
\end{proof}

\begin{definition}\label{sds}\cite{dedndreaj2021activities}
    Let $m,n$ be two natural numbers such that $1\leq m\leq \binom{n}{2}$. Expansion \eqref{expansion1} satisfying conditions (1)-(3) in Lemma \eqref{exp1lemma} is called subsequent decreasing summation decomposition of $m$ with base $n$ or shortly $sds(m,n)$. The number $k$ is called depth of the summation \eqref{expansion1} or the depth of $sds(m,n)$.
\end{definition}

\begin{remark}\label{solution}
From Definition \ref{sds} we have that the depth $k$ of $sds(m,n)$ is the smallest number in the set $\{1,\ldots,n-1\}$ which satisfies

\begin{equation}\label{rightineq}
    m\leq \frac{k(2n-k-1)}{2}.
\end{equation}

 Therefore the value for $k$ is

\begin{equation}\label{secondsolution}
    \left\lceil n - \frac{1}{2} -\sqrt{\frac{1}{4}+n^2-n-2m}\right\rceil.
\end{equation}

The value for $p_k$  is 
 \begin{equation}
        m-\sum\limits_{i=1}^{k-1}n-i
    =m-\frac{(k-1)(2n-k)}{2}.
\end{equation}

\end{remark}



Let $L(n,m)$ be a lex graph on $n$ vertices and $m$ edges such that $m\geq 1$. The depth of $sds(m,n)$ is called the depth of the lex graph $L(n,m)$. Let $N(i)$ denote the open neighborhood of the vertex $i$. For $1\leq i\leq n$ let $A(i)=\{v\in [n]: v<i\}$ and $B=\{k+1,\ldots,k+p_k\}$ then by virtue of the definition of $L(n,m)$ we have 

$$
N(i)=
\begin{cases}
V\setminus\{i\} & \text{ for } i\in\{1,\ldots,k-1\},\\
A(i)\cup B &\text{ for } i=k,\\
A(k+1)     &\text{ for } i\in \{k+1,\ldots,k+p_k\},\\
A(k) &\text{ for } i\in\{k+p_k+1,\ldots,n\}.
\end{cases}
$$

For given $m,n$ where  $1\leq m\leq \binom{n}{2}$, the sequence ${\{p_i\}_{i=1}^k}$, the numbers $p_k$ and $k$ as in Remark \ref{solution} will be crucial for the derivation of the sharp upper bounds for the number of independent sets of fixed size and the independence number for graphs of order $n$ and size $m$.

\begin{theorem}\label{bounds}
    Let $L(n,m)$ be the lex graph with $n$ vertices, $m$ edges and depth $k$. Then the maximum independent sets of $L(n,m)$ have size $n-k$.
\end{theorem}

\begin{proof}
     Let $A=\{1,\ldots,k-1\}$, $B=\{k,p_k+1,\ldots,n\}$ and $C=\{k+1,\ldots,n\}$. Notice that $A\cup B\cup C=V$ and $B^{\complement}\cap C^{\complement} \subset A$. Also, notice that no element of $A$ can belong to an independent set with cardinality greater than one because for such elements we have $N(i)=V\setminus\{i\}$. We claim that every independent set with cardinality greater than one is a subset of $B$ or a subset of $C$. Assume the opposite, that there exists an independent set $I$ with cardinality greater than one which is not a subset of $B$ nor a subset of $C$. This means that
        \begin{equation}
            I\cap B^{\complement}\neq \emptyset \text{ and } I\cap C^{\complement}\neq \emptyset,
            \end{equation}
        therefore, $I\cap A\neq \emptyset$. As a consequence, there exists an element of $A$ which belongs to an independent set with cardinality greater than one, a contradiction. 
        
        Next we will show that the sets $B, C$ are independent dominating sets and therefore maximal independent sets. That they are independent it is clear from the construction of the graph $L(n,m)$. The set
        $B$ is dominating because every element of $B^{\complement}$ is adjacent to an element of $B$. In the same way we conclude that the set $C$ is dominating.

        Note that $|B|=n-k-p_k+1$ and $|C|=n-k$. If $p_k>1$, the set $C$ is the only maximum independent set of the graph $L(n,m)$. If $p_k=1$ then $|B|=|C|$ and the graph has two maximum independent sets, $B$ and $C$.
\end{proof}

Since the lex graph $L(n,m)$ maximizes the number of independent sets of fixed size for $G(n,m)$ then $n-k$ is a sharp upper bound for the independence number $\alpha$ of $G(n,m)$-graphs. Therefore by combining Remark \ref{solution} and Theorem \ref{bounds} we have the following theorem. This result was found first by Hansen and Zheng \cite{hansen1993sharp}. We present a new and simple proof here.

    \begin{theorem}
    Let $G$ be a graph with $n$ vertices and $m$ edges and with independence number $\alpha$ then 
    \begin{equation}
        \alpha\leq \left\lfloor \frac{1}{2} +\sqrt{\frac{1}{4}+n^2-n-2m}\right\rfloor
    \end{equation}
    and the bound is sharp.
    \end{theorem}

\begin{proof}
     Note that for real numbers $a$ and $b$ we have: $a-(b+1)<a-\lceil b\rceil\leq a-b$. This imples that
    \begin{align}
        -\frac{1}{2} +\sqrt{\frac{1}{4}+n^2-n-2m}&< n -\left\lceil n - \frac{1}{2} -\sqrt{\frac{1}{4}+n^2-n-2m}\right\rceil\\[+3mm]
        &\leq  \frac{1}{2} +\sqrt{\frac{1}{4}+n^2-n-2m}
    \end{align}
    and therefore by the definition of the floor function
    
    \begin{equation}
        n -\left\lceil n - \frac{1}{2} -\sqrt{\frac{1}{4}+n^2-n-2m}\right\rceil=\left\lfloor \frac{1}{2} +\sqrt{\frac{1}{4}+n^2-n-2m}\right\rfloor.
    \end{equation}
    
    By Theorem \ref{bounds} we have that $\alpha \leq n-k$. By inserting the solution for $k$ from  Remark \ref{solution} we have:
    
    \begin{align*}
        \alpha 
        &\leq n -\left\lceil n - \frac{1}{2} -\sqrt{\frac{1}{4}+n^2-n-2m}\right\rceil\\[+3mm]
        &= \left\lfloor \frac{1}{2} +\sqrt{\frac{1}{4}+n^2-n-2m}\right\rfloor.
    \end{align*}
\end{proof}

\begin{theorem}\label{fixedcount}
    Let $L(n,m)$ be the lex graph with $n$ vertices, $m$ edges and depth $k$. Then the number of independent sets of size $r>2$ in $L(n,m)$, denoted $i_r$ is
        $$\binom{n-k-p_k}{r-1}+\binom{n-k}{r}.$$
\end{theorem}

\begin{proof}
    There are two types of independent sets of size $r$:
    \begin{enumerate}
        \item Independent sets which contain the vertex $k$. There are 
        \begin{equation}
            \binom{n-k-p_k}{r-1}
        \end{equation}
        such sets.
        \item Independent sets which do not contain the element $k$. There are 
        \begin{equation}
            \binom{n-k}{r}
        \end{equation}
        such sets.
    \end{enumerate}
\end{proof}

By using the extremality of $L(n,m)$, Theorem \ref{fixedcount} and Remark \ref{solution}  we have the following corollary which gives a sharp upper bound for the number of independent sets of fixed size.

    \begin{corollary}\label{corollary}
    Let $G(n,m)$ be a graph  with $n$ vertices and $m$ edges and let $i_r$ be its number of independent sets of size $r$ then
        \begin{equation}\label{lexbound}
          i_r\leq  \binom{n-k-p_k}{r-1}+\binom{n-k}{r}
        \end{equation}
        where $k$ and $p_k$ are as in Remark \ref{solution}
        and the bound is sharp.
    \end{corollary}
\

\begin{remark}
When $m=\sum\limits_{i=1}^kn-i$ then $n-k-p_k=0$ and the Inequality \eqref{lexbound} becomes
\begin{equation}
    i_r\leq \binom{n-k}{r}= \binom{n-k-1}{r}+ \binom{n-k-1}{r-1}.
\end{equation}
\end{remark}

Note that 
\begin{enumerate}
    \item $m=\sum\limits_{i=1}^kn-i \Longleftrightarrow \binom{n}{2}-m=\binom{s}{2}+s$ and
    \item $\left\{\sqrt{\frac{1}{4}+n^2-n-2m}\right\}=\frac{1}{2} \Longleftrightarrow \binom{n}{2}-m=\binom{s}{2}+s$ 
\end{enumerate}
where $k$ and $s$ are some natural numbers. By making use of these observations and Lemma \ref{lemma} we can see that Theorem \ref{indbound} and Theorem \ref{sproperty} are consequences of Corollary \ref{corollary}.

\printbibliography
\end{document}